\documentclass[11pt]{amsart}
\usepackage{amsfonts, amsthm, amsmath, amssymb}
\usepackage[all]{xy}
\usepackage[top=1in,bottom=1in,left=1in,right=1in]{geometry}

\def\C{\mathbb{C}}
\def\HH{\mathfrak{H}}
\def\M{\mathcal{M}}
\def\MM{\mathbb{M}}

\def\Q{\mathbb{Q}}
\def\R{\mathbb{R}}
\def\Z{\mathbb{Z}}

\newcommand{\Hom}{\operatorname{Hom}}

\newcommand{\lcm}{\operatorname{lcm}}

\newtheorem{thm}{Theorem}[section]

\newtheorem{prop}[thm]{Proposition}
\newtheorem{lem}[thm]{Lemma}
\newtheorem{cor}[thm]{Corollary}

\theoremstyle{definition}
\newtheorem*{defn}{Definition}

\usepackage[pdftex]{hyperref}

\begin{document}

\title{Generalized moonshine I: Genus-zero functions}
\author{Scott Carnahan}

\begin{abstract}
We introduce a notion of Hecke-monicity for functions on certain moduli spaces associated to torsors of finite groups over elliptic curves, and show that it implies strong invariance properties under linear fractional transformations.  Specifically, if a weakly Hecke-monic function has algebraic integer coefficients and a pole at infinity, then it is either a holomorphic genus-zero function invariant under a congruence group or of a certain degenerate type.  As a special case, we prove the same conclusion for replicable functions of finite order, which were introduced by Conway and Norton in the context of monstrous moonshine.  As an application, we introduce a class of Lie algebras with group actions, and show that the characters derived from them are weakly Hecke-monic.  When the Lie algebras come from chiral conformal field theory in a certain sense, then the characters form holomorphic genus-zero functions invariant under a congruence group.
\end{abstract}

\maketitle

\tableofcontents

\subsubsection*{Introduction}

We define a holomorphic genus-zero function to be a holomorphic function $f: \HH \to \C$ on the complex upper half-plane, with finite order poles at cusps, such that there exists a discrete group $\Gamma_f \subset SL_2(\R)$ for which $f$ is invariant under the action of $\Gamma_f$ by M\"obius transformations, inducing a dominant injection $\HH/\Gamma_f \to \C$.  A holomorphic genus-zero function $f$ therefore generates the field of meromorphic functions on the quotient of $\HH$ by its invariance group.  In this paper, we will be interested primarily in holomorphic congruence genus-zero functions, namely those $f$ for which $\Gamma(N) \subset \Gamma_f$ for some $N>0$.  These functions are often called Hauptmoduln.

The theory of holomorphic genus-zero modular functions began with Jacobi's work on elliptic and modular functions in the early 1800s, but did not receive much attention until the 1970s, when Conway and Norton found numerical relationships between the Fourier coefficients of a distinguished class of these functions and the representation theory of the largest sporadic finite simple group $\MM$, called the monster.  Using their own computations together with work of Thompson and McKay, they formulated the monstrous moonshine conjecture, which asserts the existence of a graded representation $V^\natural = \bigoplus_{n \geq -1} V_n$ of $\MM$, such that for each $g \in \MM$, the graded character $T_g(\tau) := \sum_{n \geq -1} \text{ Tr }(g|V_n)q^n$ is a normalized holomorphic genus-zero function invariant under some congruence group $\Gamma_0(N)$, where the normalization indicates a $q$-expansion of the form $q^{-1} + O(q)$.  More precisely, they gave a list of holomorphic genus-zero functions $f_g$ as candidates for $T_g$, whose first several coefficients arise from characters of the monster, and whose invariance groups $\Gamma_{f_g}$ contain some $\Gamma_0(N)$ \cite{CN79}.  By unpublished work of Koike, the power series expansions of $f_g$ satisfy a condition known as complete replicability, given by a family of recurrence relations, and the relations determine the full expansion of $f_g$ from only the first seven coefficients of $f_{g^n}$ for $n$ ranging over powers of two.

This conjecture was proved in \cite{B92} using a combination of techniques from the theory of vertex algebras and infinite dimensional Lie algebras: $V^\natural$ was constructed in \cite{FLM88} as a vertex operator algebra, and Borcherds used it to construct the monster Lie algebra, which inherits an action of the monster.  Since the monster Lie algebra is a generalized Kac-Moody algebra with a homogeneous action of $\MM$, it admits twisted denominator formulas, which relate the coefficients of $T_g$ to characters of powers of $g$ acting on the root spaces.  In particular, each $T_g$ is completely replicable, and Borcherds completed the proof by checking that the first seven coefficients matched the expected values.

Knowing this theorem and some additional data, one can ask at least two natural questions:
\begin{enumerate}
\item The explicit checking of coefficients at the end of the proof has been called a ``conceptual gap'' in \cite{CG97}, and this problem has been rectified in some sense by replacing that step with non-computational theorems:
\begin{enumerate}
\item[(i)] In \cite{B92}, it was pointed out that the twisted denominator formulas imply that the functions $T_g$ are completely replicable.
\item[(ii)]  \cite{K94} has a proof that completely replicable functions satisfy lots of modular equations.
\item[(iii)] \cite{CG97} has a proof that power series satisfying enough modular equations are either holomorphic genus-zero and invariant under $\Gamma_0(N)$, or of a particular degenerate type resembling trigonometric functions.
\item[(iv)] One can eliminate the degenerate types, either by appealing to a result in \cite{M94} asserting that completely replicable series that are ``$J$-final'' (a condition that holds for all $T_g$, since $T_1=J$) are invariant under $\Gamma_0(N)$ for some $N$, or by using a result in \cite{DLM00} that restricts the form of the $q$-expansions at other cusps.
\end{enumerate}
Since modular functions live on moduli spaces of structured elliptic curves, one might ask, how do these recursion relations and replicability relate to group actions and moduli of elliptic curves?
\item One might wonder if similar behavior applies to groups other than the monster.  It was suggested in \cite{CN79} that other sporadic groups may exhibit properties resembling moonshine, and \cite{Q81} produced strong computational evidence for this.  These data were organized in \cite{N87} into the generalized moonshine conjecture, which asserts the existence of a generalized character $Z$ that associates a holomorphic function on $\HH$ to each commuting pair of elements of the monster, satisfying the following conditions:
\begin{enumerate}
\item $Z(g,h,\tau)$ is invariant under simultaneous conjugation of $g$ and $h$.
\item For any $\binom{ab}{cd} \in SL_2(\Z)$, there exists a nonzero constant $\gamma$ (said to be a $24$th root of unity in \cite{N01}) such that $Z(g^ah^c,g^bh^d, \tau) = \gamma Z(g,h,\frac{a\tau+b}{c\tau+d})$.
\item The coefficients of the $q$-expansion of $Z(g,h,\tau)$ for fixed $g$ form characters of a graded representation of a central extension of $C_\MM(g)$.
\item $Z(g,h,\tau)$ is either constant or holomorphic congruence genus-zero.
\item $Z(g,h,\tau) = j(\tau)-744 = q^{-1} + 196884q + 21493760q^2 + \dots$ if and only if $g=h=1$.
\end{enumerate}
This conjecture is still open, but if we fix $g=1$, it reduces to the original moonshine conjecture.  One might hope that techniques similar to those used in \cite{B92} can be applied to attack this conjecture in other cases, and the answer seems to be affirmative.  For example, \cite{H03} affirms the cases where $g$ is an involution in conjugacy class 2A, using a construction of a vertex algebra with baby monster symmetry, and roughly following the outline of Borcherds's proof.  However, there are obstructions to making this technique work in general, since there are many elements of the monster for which we do not know character tables of centralizers or their central extensions.  One might ask, is there a reasonably uniform way to generate holomorphic congruence genus-zero functions from actions of groups on certain Lie algebras?
\end{enumerate}

This paper is an attempt to unify the two questions, and set the stage for a more detailed study of the infinite dimensional algebraic structures involved.  The main result is that modular functions (and more generally, singular $q$-expansions with algebraic integer coefficients) that are holomorphic on $\HH$ and satisfy a certain Hecke-theoretic property are holomorphic congruence genus-zero or degenerate in a specified way.  As a special case, we find that finite order replicable functions with algebraic integer coefficients, as defined in section 4, satisfy the same property.  The algebraic integer condition is sufficient for our purposes, since we intend to use this theorem in the context of representations of finite groups.  Since modular functions with algebraic coefficients that are holomorphic on $\HH$ and invariant under a congruence group have bounded denominators (see Theorem 3.52 of \cite{S71} and divide by a suitable power of $\Delta$), it is reasonable to conjecture that all holomorphic congruence genus-zero functions whose poles have integral residue and constant term have algebraic integer coefficients.

We apply the theory to show that when a group acts on an infinite dimensional Lie algebra with a special form, the character functions are holomorphic congruence genus-zero.  We call these algebras ``Fricke compatible'' because they have the form we expect from elements $g \in \MM$ for which the function $T_g$ is invariant under a Fricke involution $\tau \mapsto -1/N\tau$.  Later papers in this series will focus on constructing these and other (``non-Fricke compatible'') Lie algebras, first by generators-and-relations, and then by applying a version of the no-ghost theorem to abelian intertwiner algebras.  At the time of writing, this strategy does not seem to yield a complete proof of generalized moonshine, because of some subtleties in computing eigenvalue multiplicities for certain cyclic groups of composite order acting on certain irreducible twisted modules of $V^\natural$.  It is possible that some straightforward method of controlling these multiplicities has escaped our attention, but for the near future we plan to rest the full result on some precisely stated assumptions.

Most of the general ideas in the proof are not new, but our specific implementation bears meaningful differences from the existing literature.  In fact, Hecke operators have been related to genus-zero questions since the beginning of moonshine, under the guise of replicability, and the question of relating replication to holomorphic genus-zero modular functions was proposed in the original paper \cite{CN79}.  However, the idea of using an interpretation via moduli of elliptic curves with torsors is relatively recent, and arrives from algebraic topology.  Equivariant Hecke operators, or more generally, isogenies of (formal) groups, can be used to describe operations on complex-oriented cohomology theories like elliptic cohomology, and they were introduced in various forms in \cite{A95} and \cite{B98}.  More precise connections to generalized moonshine were established in \cite{G07}.

\subsubsection*{Summary} In the first section, we introduce Hecke operators, first as operators on modular functions, and then on general power series.  In section 2, we define Hecke-monicity and prove elementary properties of Hecke-monic functions.  In section 3, we relate Hecke-monicity to equivariant modular equations.  Most of this step is a minor modification of part of Kozlov's master's thesis \cite{K94}.  In section 4, we prove a holomorphic congruence genus-zero theorem, and our proof borrows heavily from \cite{CG97}.  Most of the arguments require minimal alteration from the form given in that paper, so in those cases we simply indicate which changes need to be made.  In section 5, we focus on the special case of replicable functions, and we show that those with finite order and algebraic integer coefficients are holomorphic congruence genus-zero or of a specific degenerate type.  In section 6, we conclude with an application to groups acting on Lie algebras, and show that under certain conditions arising from conformal field theory, the characters from the action on homology yield holomorphic congruence genus-zero functions.

\subsubsection*{Acknowledgments} The author thanks Richard Borcherds, Gerald Hoehn, Jacob Lurie, and Arne Meurman.  Borcherds suggested generalized moonshine as a dissertation project, and gave much useful advice and perspective.  Hoehn offered many helpful comments on an earlier draft, from which this paper was drawn.  Lurie provided inspiring conversations, and suggested that twisted denominator formulas appear to be constructed from equivariant Hecke operators.  Meurman kindly mailed a copy of Kozlov's masters thesis \cite{K94} across the Atlantic Ocean.  The author would also like to thank the referees for many helpful suggestions and corrections, and an anonymous member of the editorial board for suggestive hints concerning complex analytic moduli.  This material is partly based upon work supported by the National Science Foundation under grant DMS-0354321.

\section{Equivariant Hecke operators}

The aims of this section are to introduce a combinatorial formula for equivariant Hecke operators for functions that are not necessarily modular, and to prove some elementary properties.  The geometric language of stacks, and torsors is used in only in this section, and only to justify the claim that these Hecke operators occur naturally.  It is not strictly necessary for understanding the formula, and the reader may skip everything in this section except for the statements of the lemmata without missing substantial constituents of the main theorem.

Let $G$ be a finite group, and let $\mathcal{M}_{Ell}^G$ denote the analytic stack of elliptic curves equipped with $G$-torsors (also known as the Hom stack $\underline{Hom}(\mathcal{M}_{Ell}, BG)$).  Objects in the fibered category are diagrams $P \to E \overset{e}{\leftrightarrows} S$ of complex analytic spaces satisfying:
\begin{enumerate}
\item $P \to E$ is a $G$-torsor (i.e., an analytically locally trivial principal $G$-bundle).
\item $E \to S$ is a smooth proper morphism, whose fibers are genus one curves.
\item $e$ is a section of $E \to S$.
\end{enumerate}
Morphisms are fibered diagrams satisfying the condition that the torsor maps are $G$-equivariant.  This is a smooth Deligne-Mumford stack (in the sense of \cite{BN06}).  For each positive integer $n$, we consider the degree $n$ $G$-Hecke correspondence, which is given by the diagram $\mathcal{I}_n^G \underset{t}{\overset{s}{\rightrightarrows}} \mathcal{M}_{Ell}^G$, defined as follows: $\mathcal{I}^G_n$ is the stack of $n$-isogenies of elliptic curves with $G$-torsors.  Its objects are diagrams
\[ \xymatrix{ & & P_2 \ar[d]^\pi \\ E_1 \ar[rr]^f \ar@<1ex>[rd] & & E_2 \ar@<-1ex>[ld] \\ & S \ar@/_/[ur]_{e_2} \ar@/^/[ul]^{e_1} } \]
of complex analytic spaces, where
\begin{enumerate}
\item $\pi$ is a $G$-torsor.
\item $E_1 \to S$ and $E_2 \to S$ are smooth proper morphisms, whose geometric fibers are genus one curves.
\item $e_1$ and $e_2$ are sections of the corresponding maps.
\item $f$ is an $n$-isogeny, i.e., a homomorphism whose kernel is a finite flat $T$-group scheme of length $n$ (in particular, $f$ makes the evident triangle diagrams commute).
\end{enumerate}
As before, morphisms are fibered diagrams satisfying the condition that the torsor maps are $G$-equivariant.  The two canonical maps $s,t: \mathcal{I}^G_n \to \mathcal{M}_{Ell}^G$ are defined by $s(E_1, E_2, P_2, S) = (E_1 \times_{E_2} P_2 \to E_1 \leftrightarrows S)$ and $t(E_1, E_2, P_2, S) = (P_2 \to E_2 \leftrightarrows S)$ for objects, and the evident diagrams are given for morphisms.  One can show that $s$ and $t$ are finite \'etale morphisms of degree $\psi(n) = \prod_{p|n} (1+1/p)$, essentially by transferring the arguments of Proposition 6.5.1 in \cite{KM85} to the analytic setting.

The Hecke operator $nT_n$ is defined as the canonical trace map $s_* t^*$ on the structure sheaf of $\mathcal{M}_{Ell}^G$.  Over each point, it satisfies the formula:
\[ nT_n(f)(P \overset{G}{\to} E) = \underset{|H|=n}{\sum_{0 \to H \to E' \overset{\pi}{\to} E \to 0}} f( \pi^*P \overset{G}{\to} E' ), \]
where the sum is over all degree $n$ isogenies to $E$.  When $G$ is trivial, this is the usual weight zero Hecke operator.

We wish to describe these operators in terms of functions on the complex upper half plane, and this requires an analytic uniformization of the moduli problem.  Following the unpublished book \cite{C?} and \cite{D69}, the upper half plane classifies pairs $(\pi, \psi)$, where $\pi: E \to S$ is an elliptic curve, and $\psi: R^1 \pi_* \underline{\mathbb{Z}} \to \underline{\mathbb{Z}^2}$ is an isomorphism whose exterior square induces the (negative of the) canonical isomorphism $R^2 \pi_* \underline{\mathbb{Z}}(1) \to \underline{\mathbb{Z}}$.  By dualizing, one has a universal diagram $\mathbb{Z}^2 \times \mathfrak{H} \to \mathbb{C} \times \mathfrak{H} \to \mathfrak{H}$ defining a family of elliptic curves $E \to \mathfrak{H}$ equipped with oriented bases of fiberwise $H_1$.  There is an $SL_2(\Z)$-action from the left via M\"obius transformations (equivalently, changing the oriented homology basis of a curve), that induces a surjection onto $\mathcal{M}_{Ell}$.  When we consider $G$-torsors over elliptic curves equipped with homology bases, we find that they are classified up to isomorphism by their monodromy along the basis, given by a conjugacy class of a pair of commuting elements in $G$.  Since a pair of commuting elements is a homomorphism from $\mathbb{Z} \times \mathbb{Z} \to G$, there is also an action of $SL_2(\mathbb{Z})$ on commuting pairs of elements from the right via $(g,h)\binom{ab}{cd} = (g^ah^c,g^bh^d)$.  We obtain an identification on the level of points:

\[ \mathcal{M}_{Ell}^G \cong \Hom(\Z \times \Z, G)/G \underset{SL_2(\Z)}{\times} \HH \]
where the quotient by $G$ arises from the action by conjugation on the target.  The identification can be promoted to an equivalence of analytic stacks by choosing a uniformizing moduli problem of triples $(P \to E \to S, \psi, \tilde{e}: S \to P)$, where $\psi$ is as above, and $\tilde{e}$ is a lift of $e: S \to E$ to the $G$-torsor.  It is represented by a disjoint union of upper half planes in bijection with $\Hom(\Z \times \Z, G)$, and one obtains the quotient via commuting actions of $G$ (on the set of lifts $\tilde{e}$) and $SL_2(\mathbb{Z})$ (on the set of $\psi)$.

With this presentation, we can recast the Hecke operators in terms of holomorphic functions on the complex upper half plane $\HH$.  We can write any $f: \mathcal{M}_{Ell}^G \to \C$ as $f(g,h,\tau)$, for $g$ and $h$ commuting elements of $G$, and $\tau \in \HH$.  $f$ is invariant under simultaneous conjuation on $g$ and $h$, and satisfies $f(g^ah^c, g^bh^d,\tau) = f(g,h,\frac{a\tau+b}{c\tau+d})$.  In particular, for fixed $g$ and $h$, $f(g,h,\tau)$ is a holomorphic modular function, invariant under $\Gamma(\lcm(|g|,|h|))$.  Following \cite{G07}, we map the homology basis to $(-1,\tau)$, so $(g,h,\tau)$ describes an elliptic curve $\C/\langle -1,\tau \rangle$ equipped with a $G$-torsor with monodromy $(g,h)$.  (Many texts use the basis $(1,\tau)$ when studying modular functions, mostly because $\tau$ then becomes the ratio of periods, but our convention is what we need for the left $SL_2(\Z)$ action to work correctly.)  Any degree $n$ isogeny from an elliptic curve $E'$ to $\C/\langle -1,\tau \rangle$ can be described as the identity map on $\C$, where $E'$ is the quotient by a unique index $n$ sublattice of $\langle -1,\tau \rangle$.  Since we are assuming $SL_2(\Z)$-equivariance of $f$, we can choose any basis, and get the same value from $f$.  We preferentially choose bases $(-d, a\tau+b)$, where $d$ is the index in $\Z$ of the intersection of the sublattice with $\Z$, and we get the following formula:

\[ T_n(f)(g,h,\tau) = \frac{1}{n} \underset{0 \leq b<d}{\sum_{ad=n}} f\left(g^d, g^{-b}h^a, \frac{a\tau+b}{d}\right) \]

Suppose we wanted to extend the notion of Hecke operator to a larger class of functions, in particular ones which are not a priori completely independent of the choice of homology basis of our elliptic curve.  One might hope that we could have a good notion for all functions on $\Hom(\Z \times \Z, G)/G \times \HH$.  Unfortunately, there is no canonical choice of homology basis for $E'$ (i.e., a basis for the index $n$ sublattice of $\langle -1,\tau \rangle$), and it is difficult make choices in a systematic way that makes the sum a canonical quantity, so we do not know of any definition of Hecke operator for arbitrary functions on $\Hom(\Z \times \Z, G)/G \times \HH$ that is particularly natural.  However, there is an intermediate form of equivariance for which we {\em can} make a canonical definition, using the subgroup $\pm \Z := \{ \pm \binom{1n}{01} | n \in \Z \} \cong \Z \times \Z/2\Z$.  Invariance under this group implies that for a function $f$ on $\Hom(\Z \times \Z, G)/G \underset{\pm \Z}{\times} \HH$, $f(g,gh,\tau) = f(g,h,\tau+1)$, so the Fourier expansion of $f(g,h,\tau)$ is a power series in $q^{1/|g|}$ that converges on the punctured open unit disc parametrized by $q^{1/|g|}$, $|q|<1$.  We will assume the existence of a lower bound on exponents, i.e., that all of our power series are Laurent series.

We can interpret this geometrically.  The quotient $\Hom(\Z \times \Z, G)/G \underset{\pm \Z}{\times} \HH$ is a disjoint union of punctured unit discs, and parametrizes $G$-torsors over elliptic curves that are equipped with a distinguished primitive element of $H_1$ (up to sign - the $-1$ automorphism inverts the monodromy and fixes the curve, so as long as we remember that any function is invariant under this transformation, we can safely ignore it).  This element functions as the first element in the homology basis, since the action of $\pm \Z$ renders all choices of second oriented basis element equivalent.  It also uniquely determines a multiplicative uniformization $\C^\times \overset{\pi}{\to} E$ with kernel $\langle q \rangle$, $|q|<1$.  One can classify the $G$-torsors over an elliptic curve with multiplicative uniformization by studying its monodromy.  Monodromy along the primitive homology element gives a distinguished element $g$, up to conjugacy.  Monodromy along a path from $1$ to $q$ in $\C^\times$ yields a commuting element $h$, that is unique up to conjugation that is simultaneous with $g$.  However, the set of homotopy classes of paths to $q$ is a $\Z$-torsor given by winding number around zero, and the action changes this monodromy by powers of $g$, so the equivalence classes of $G$-torsors are determined by assigning a commuting element $h$ not to $q$, but to a choice of $q^{1/|g|}$.

\begin{defn} Given an elliptic curve $E$ equipped with a multiplicative uniformization, a restricted degree $n$ isogeny is a pullback diagram:

\[ \xymatrix{
\C^\times \ar[r] \ar[d] & \C^\times \ar[d] \\ E' \ar[r] & E
} \]

where the bottom row is a degree $n$ isogeny of elliptic curves.
\end{defn}

The map on top is then given by a $d$th power map, for some $d|n$.  If we examine kernels of the uniformization, we find that this induces an inclusion $\Z \to \Z$ by multiplication by $a := n/d$.  If we let $q$ generate the kernel of the uniformization on the target, the isogenies that pull back to the $d$th power map on $\C^\times$ are then classified by $d$th roots of $q^a$ in the source, and there are exactly $d$ of them.  In particular, there is a bijection between degree $n$ isogenies in the classical sense and degree $n$ restricted isogenies.  Each restricted isogeny then has the form

\[ \xymatrix{ \C^\times/q^{\frac{a}{d}\Z} \ar[r]^{z \mapsto z^d} & \C^\times/q^\Z } \]

We can rephrase this using lattices: A uniformized elliptic curve is given by an equivalence class of lattices $\langle -1, \tau \rangle$, where we consider two lattices equivalent if the second elements differ by an integer.  The isogeny condition is equivalent to demanding that the distinguished homology basis element is $-d$ for some $d|n$, as we chose before.

We can now define our Hecke operators, by summing over pullbacks along our restricted isogenies.
\begin{lem} Given a function $f$ on $\Hom(\Z \times \Z, G)/G \underset{\pm \Z}{\times} \HH$, define the function $n\hat{T}_nf$ on the same space by assigning to each elliptic curve equipped with a $G$-torsor and multiplicative uniformization the sum of $f$ evaluated on the sources of restricted isogenies of degree $n$.  Then:

\[ n\hat{T}_nf(g,h,\tau) = \underset{0 \leq b<d}{\sum_{ad=n}} f\left(g^d, g^{-b}h^a, \frac{a\tau+b}{d}\right), \]
i.e., we get the same formula for Hecke operators as we would over $\M^G_{Ell}$.
\end{lem}
\begin{proof}
Fix an elliptic curve $E$, with a multiplicative uniformization and a $G$-torsor.  We may assume that $E \cong \C/\langle -1, \tau \rangle$ for some $\tau \in \HH$, where the path from $0$ to $-1$ along $\R$ maps to the distinguished homology element.  Let $g$ be the monodromy of the $G$-torsor along the image of this path, and let $h$ be the monodromy along the image of a path from $0$ to $\tau$.  Fix a restricted isogeny, i.e., an index $n$ sublattice of $\langle -1, \tau \rangle$ together with a (uniquely defined) fixed negative integer $-d$, $d|n$.  The path from $0$ to $-d$ is the chosen primitive homology element of the source elliptic curve.  The monodromy of the $G$-torsor along the image of this path is $g^d$.  The sublattice is characterized by a second homology generator $a\tau+b$ for $a = n/d$, and $b$ is uniquely determined modulo $d$.  The generator then has monodromy $h^ag^{-b}$, and by applying the $\frac{1}{d}$-dilation homothety, the elliptic curve is given by the point $\frac{a\tau+b}{d}$.  To show that the formula above holds, it suffices to show that $f$, evaluated on these generators, does not depend on which coset representative modulo $d$ we choose.  This independence arises from the $\pm \Z$-equivariance, i.e., if we choose $b$ and $b'$ such that $b-b' = kd$, then
\[ \begin{aligned}
f\left(g^d, g^{-b'}h^a, \frac{a\tau+b'}{d}\right) &= f\left(g^d, g^{-b}h^ag^{kd}, \frac{a\tau+b - kd}{d}\right) \\
&= f\left(g^d, g^{-b}h^a, \frac{a\tau+b}{d} - k + k\right) \\
&= f\left(g^d, g^{-b}h^a, \frac{a\tau+b}{d}\right)
\end{aligned} \]
\end{proof}

From now on, we will use the notation $nT_n$ for this Hecke operator, instead of $n\hat{T}_n$.

\begin{lem}
If $f$ is a function on $\Hom(\Z \times \Z, G)/G \underset{\pm \Z}{\times} \HH$, then
\[ T_k T_m f(g,h,\tau) = \sum_{t|(k,m)} \frac{1}{t} T_{km/t^2} f(g^t, h^t, \tau). \]
\end{lem}
\begin{proof}
\[ \begin{aligned}
T_k T_m f(g,h,\tau) &= T_k \frac{1}{m} \underset{(a,b,d)=1}{\underset{0 \leq b<d}{\sum_{ad=m}}} f\left(g^d, g^{-b}h^a, \frac{a\tau+b}{d}\right) \\
&= \frac{1}{km} \underset{0 \leq b' < d'}{\sum_{a'd' = k}} \underset{0 \leq b < d}{\sum_{ad=m}} f\left(g^{dd'}, g^{-bd'-ab'}h^{aa'}, \frac{aa' \tau + ab' + bd'}{dd'}\right) \\
&= \frac{1}{km} \underset{t=(a,d')}{\underset{ad=m}{\sum_{a'd' = k}}} \underset{0 \leq b < d}{\sum_{0 \leq b' < d'}} f\left(g^{t\frac{dd'}{t}}, g^{t\frac{-bd'-b'a}{t}}h^{t\frac{aa'}{t}}, \frac{a'\frac{a}{t} \tau + b'\frac{a}{t} + b\frac{d'}{t}}{d\frac{d'}{t}}\right) \\
&= \frac{1}{km} \sum_{t|(k,m)} \underset{(a,d')=1}{\underset{ad=\frac{m}{t}}{\sum_{a'd' = \frac{k}{t}}}} \underset{0 \leq b < d}{\sum_{0 \leq b' < td'}} f\left(g^{tdd'}, g^{t(-bd'-ab')}h^{taa'}, \frac{aa' \tau + ab' + bd'}{dd'}\right) \\
&= \frac{1}{km} \sum_{t|(k,m)} \sum_{a''d''=km/t^2} t \sum_{0 \leq b'' < d''} f\left(g^{td''}, g^{-tb''}h^{ta''}, \frac{a'' \tau + b''}{d''}\right) \\
&= \sum_{t|(k,m)} \frac{1}{t} T_{km/t^2} f(g^t, h^t, \tau)
\end{aligned} \]
We shall explain the second to last equality using an argument from \cite{K94}.  In this step, we substitute $a'' = aa'$, $d'' = dd'$, and $b''$ for any solution to the congruence $ab' + bd' \equiv b''$ (mod $dd'$).  By $\pm \Z$-invariance, it remains to show that for any $0 \leq b'' < dd'$, this congruence has exactly $t$ solutions.  There are exactly $tdd'$ possible values of $b$ and $b'$ satisfying $0 \leq b < d, 0 \leq b' < td'$, and $dd'$ values of $b''$ satisfying $0 \leq b'' < d''$.  The first value is $t$ times the second value, so it suffices to show that for any fixed admissible pair $(b,b')$ there are exactly $t$ solutions $(c, c')$ satisfying $0 \leq c < d$ and $0 \leq c' < td'$ to the congruence $ab' + bd' \equiv ac' + cd'$ (mod $dd'$).  Any such solution yields the identity $dd' | a(b'-c') + d'(b-c)$, so $d'|a(b'-c')$.  Since $(a,d') = 1$, $d'|b'-c'$, so we write $c' = b' + sd'$, and there are $t$ choices of $s$ that satisfy $0 \leq c' < td'$.  Cancelling $d'$ in the identity yields $d|as + b-c$, so each choice of $s$ gives a uniquely defined value of $c$ satisfying $0 \leq c < d$.
\end{proof}

It is also possible to prove this by working one prime at a time, or by invoking the moduli interpretation and enumerating restricted isogenies.

\section{Hecke-monicity}

\begin{defn} Let $f$ be a holomorphic function on $\Hom(\Z \times \Z, G)/G \underset{\pm \Z}{\times} \HH$.  We say that $f$ is Hecke-monic if on each connected component, the restriction of $nT_n(f)$ is a monic polynomial of degree $n$ in the restriction of $f$, for all positive integers $n$.
\end{defn}

\noindent\textbf{Remark:} Since we only require our functions to admit translation-equivariance, and the Hecke operators only involve transformations of the form $\tau \mapsto \frac{a\tau + b}{d}$, Hecke-monicity only depends on the values of $f$ when the monodromy around the first homology basis element lies in a subset of $G$ that is closed under taking power maps.  We will find it useful to weaken the condition that $f$ be defined on all components.  For example, if we choose $g \in G$, we only need to consider the functions $\{ f(1,g^i,\tau) \}_{i>0}$ to define Hecke operators on $f(1,g,\tau)$.

\begin{defn} Let $g,h \in G$ be commuting elements, and let $f$ be a function on the connected components of $\Hom(\Z \times \Z, G)/G \underset{\pm \Z}{\times} \HH$ corresponding to pairs $(g^d, g^{-b} h^a)$ for $a,b,d >0$.  We say that $f$ is weakly Hecke-monic for $(g,h)$ if for all $n>0$, $nT_nf(g,h,\tau)$ is a monic polynomial of degree $n$ in $f(g,h,\tau)$.  We say that $f$ is semi-weakly Hecke-monic for $(g,h)$ if for all $n>0$, $nT_nf(g^d,g^{-b} h^a,\tau)$ is a monic polynomial of degree $n$ in $f(g^d,g^{-b} h^a, \tau)$ for all $a,b,d>0$.
\end{defn}

We will use the notation $e(x)$ to denote $e^{2 \pi i x}$ for the rest of this paper.

\begin{lem} \label{lem:form}
Let $f$ be a weakly Hecke-monic function for $(g,h)$, and let $N>0$ satisfy $g^N = h^N = 1$.  If $f(g,h,\tau)$ has a singularity at infinity, then its $q$-expansion has the form $\zeta q^{C/|g|} + O(1)$ for $C$ a negative integer and $\zeta$ a root of unity satisfying $\zeta^N = 1$ if $N$ is even and $\zeta^{2N} = 1$ if $N$ is odd.
\end{lem}
\begin{proof}
Let $f(g,h,\tau) = \sum_{n \in \frac{1}{|g|}\Z} a_n q^n = a_{n_0}q^{n_0} + a_{n_1}q^{n_1} + \dots$ for $a_{n_0}$ nonzero, $n_0 < 0$, and let $p$ be a prime congruent to $1$ mod $N$.  Then
\[ \begin{aligned}
pT_pf(g,h,\tau) &= f(g,h^p,p\tau) + \sum_{b=0}^{p-1} f(g^p,g^{-b}h,\frac{\tau + b}{p}) \\
&= f(g,h,p\tau) + \sum_{b=0}^{p-1} f(g,h, \frac{\tau+b}{p}-b) \\
&= \sum_n a_n q^{pn} + \sum_n a_n \sum_b e(n(\frac{\tau + b}{p} - b)) \\
&= \sum_n a_n q^{pn} + \sum_n a_n q^{n/p} \sum_b e(nb\frac{1-p}{p})
\end{aligned} \]
Hecke-monicity implies that $a_{n_0}q^{pn_0} = (a_{n_0}q^{n_0})^p$ for all $p$ congruent to $1$ mod $N$, so $\zeta = a_{n_0}$ is an $M$th root of unity, where $M$ is the greatest common divisor of $\{ p-1 | p \text{ prime, } p \equiv 1 (N) \}$.  $M = kN$ for some integer $k$, and $(k,N) = 1$, since otherwise $\frac{k}{(k,N)}N + 1$ would be a residue class mod $M$ that is coprime to $M$ but not congruent to 1.  By the Chinese Remainder Theorem, $k$ must have a unique residue class mod $k$ that is coprime to $k$.  Therefore, the only possible values of $k$ are 1 or 2, and $k=2$ is only possible when $N$ is odd.

If $f(g,h,\tau)$ is a singular monomial $\zeta q^{C/|g|}$ with $C<0$, then we are done.  Otherwise, we assume $a_{n_1} \neq 0$, and from the calculation above, we have:
\[ pT_pf(g,h,\tau) = \begin{cases} a_{n_0}q^{pn_0} + a_{n_0}q^{n_0/p} \sum_b e(n_0 b\frac{p-1}{p}) + \dots & n_1 > n_0/p^2 \\ a_{n_0}q^{pn_0} + a_{n_1}q^{pn_1} + \dots & n_1 < n_0/p^2 \end{cases} \]
We will not bother with the case of equality, because we will let $p$ become large.  If $n_1 < 0$, then the second case will hold for almost all $p$ congruent to $1$ mod $N$, and if $n_1 \geq 0$, then the first case will hold for all such $p$.   If $n_1 < 0$ and $p$ is sufficiently large, then
\[ \begin{aligned}
a_{n_0}q^{pn_0} + a_{n_1}q^{pn_1} + \dots &= (a_{n_0}q^{n_0} + a_{n_1}q^{n_1} + \dots)^p + c(a_{n_0}q^{n_0} + \dots)^{p-1} + \dots \\
&= a_{n_0} q^{pn_0} + pa_{n_0}^{p-1}a_{n_1} q^{(p-1)n_0 + n_1} + \dots \\
&= a_{n_0}q^{pn_0} + pa_{n_1}q^{(p-1)n_0 + n_1} + \dots
\end{aligned} \]
This yields an equality $a_{n_1}q^{pn_1} = pa_{n_1}q^{(p-1)n_0 + n_1}$, which under our assumptions is a contradiction.  Therefore, $n_1 \geq 0$, and we are done.
\end{proof}

\begin{lem} \label{lem:asymp}
Let $f$ be a weakly Hecke-monic function for $(g,h)$, such that $f(g,h,\tau) = \zeta q^{C/|g|} + O(1)$ for some $C<0$ and some root of unity $\zeta$.  Then there exists some $N$ such that $pT_pf(g,h,\tau) = \zeta^p q^{Cp/|g|} + O(1)$ for all primes $p > N$.
\end{lem}
\begin{proof}
Since $\langle g,h \rangle$ has finite order, we can choose $N$ such that $N/|g|$ is greater than the order of any pole of $f(g^k,g^l h^m,\tau)$ at infinity, as $k$ and $l$ range over $\Z/|g|\Z$ and $m$ ranges over $\Z/|h| \Z$.  Suppose $p>N$, and write any singular functions $f(g^p,g^{-b}h,\tau)$ as $\zeta_b q^{C_b/|g|} + O(1)$.  Then:
\[ \begin{aligned}
pT_pf(g,h,\tau) &= f(g,h^p,p\tau) + \sum_{0 \leq b < d} f(g^p,g^{-b}h,\frac{\tau+b}{p}) \\
&= f(g,h^p,p\tau) + \sum_b \zeta_b e(\frac{bC_b}{p|g|}) q^{C_b/p|g|} + O(1)
\end{aligned} \]
Since $p>|C_b|$, $C_b/p|g| > -1/|g|$ for all $b$ such that $f(g^p,g^{-b}h,\tau)$ has a pole at infinity.  However, the above sum is a polynomial in $f(g,h,\tau)$, and therefore a power series in $q^{1/|g|}$, so the only contribution with a negative power of $q$ comes from $f(g,h^p,p\tau)$.  We have $f(g,h^p,p\tau) = \zeta' q^{C'p/|g|} + O(1)$ for some $\zeta'$ and $C'$, and since this is a monic polynomial of degree $p$ in $f(g,h,\tau)$, we have $\zeta' = \zeta^p$ and $C'=C$.
\end{proof}

\begin{prop} \label{prop:expansion}
Let $f$ be a weakly Hecke-monic function for $(g,h)$, such that $f(g,h,\tau) = \zeta q^{C/|g|} + O(1)$ for some $C<0$ and some root of unity $\zeta$.  Then $f(g,h,\tau)$ is invariant under translation by $|g|/C$, i.e., the only nonzero terms in the $q$-expansion are those with integer powers of $q^{C/|g|}$.
\end{prop}

\begin{proof}
Suppose $f(g,h,\tau)$ is not a power series in $q^{C/|g|}$, and let $n_0$ be the smallest integer such that $n_0$ is not a multiple of $C$, and the coefficient $a_{n_0}$ of $q^{n_0/|g|}$ in the $q$-expansion of $f(g,h,\tau)$ is nonzero.  Choose $N$ as in Lemma \ref{lem:asymp}, and let $p$ be a prime satisfying $p>N$, $p \equiv 1$ (mod $|g|$) and $(p-1)C + n_0 < 0$ (i.e., $p$ is large).

By the above lemma, $pT_pf(g,h,\tau)$ has $q$-expansion $\zeta^p q^{Cp/|g|} + O(1)$.  However, $pT_pf(g,h,\tau)$ is a monic polynomial of degree $p$ in $f(g,h,\tau)$, so we can write its $q$-expansion as a sum of a series in $q^{C/|g|}$ and a series with initial term $p\zeta^{p-1}a_{n_0} q^{(p-1)C+n_0}$.  Since the coefficient is nonzero and the exponent is negative, we have a contradiction.
\end{proof}

\section{Modular equations}

Cummins and Gannon found a characterization of holomorphic genus-zero functions invariant under $\Gamma_0(N)$ as power series satisfying many modular equations.  We show that weakly Hecke-monic functions satisfy a similar condition, and modify the first half of their proof to get global symmetries.  In particular, any Hecke-monic function on $\mathcal{M}_{Ell}^G$ is holomorphic congruence genus-zero on nonconstant components.

\begin{lem}
Fix a positive integer $n$, and let $f$ be weakly Hecke-monic for $(g^t,h^t)$ for all $t|n$.  Then the power sum symmetric polynomials in
\[ \left\{ f(g^d, g^{-b}h^a, \frac{a\tau+b}{d}) | ad=n, 0 \leq b < d \right\} \]
are polynomials in $f(g^t, h^t, \tau)$ for $t$ ranging over positive integers dividing $n$.  Furthermore, the term with highest degree in $f(g,h,\tau)$ has coefficient equal to one.  In particular, if $n$ is a prime satisfying $g^n = g$ and $h^n=h$, then the power sums are polynomials in $f(g,h,\tau)$.
\end{lem}
\begin{proof}
This is essentially the same as in \cite{K94}.  We apply $nT_n$ to the equation $f^m = mT_m(f) - a_{m-1}f^{m-1} - \dots - a_1f - a_0$ to find that the power sum
\[ \underset{0 \leq b<d}{\sum_{ad=n}} f\left(g^d, g^{-b}h^a, \frac{a\tau+b}{d}\right)^m = n T_n(f(g,h,\tau)^m)\]
can be written as a sum of $mn T_n T_m (f)(g,h,\tau)$ and a linear combination of $T_n$ applied to lower degree polynomials in $f(g,h,\tau)$.  By induction on $m$, these are polynomials in $f(g^t, h^t,\tau)$ for $t|n$.
\end{proof}

\begin{lem}
Fix $n \geq 2$ squarefree, and let $f$ be a weakly Hecke-monic function for $(g^t,h^t)$ for all $t|n$.  Then there exists a monic polynomial $F_n(x)$ of degree $n \prod_{p|n} \frac{p+1}{p}$, whose coefficients are polynomials in $f(g^t, h^t, \tau)$ for $t|n$, and $F_n(x)$ has roots $\left\{ f(g^d, g^{-b}h^a, \frac{a\tau+b}{d}) | ad=n, 0 \leq b < d, (a,b,d)=1 \right\}$ for any $\tau$.
\end{lem}
\begin{proof}
Since $n$ is squarefree, the condition $(a,b,d) = 1$ is a consequence of $ad=n$.  The power sums generate the ring of symmetric polynomials in
\[ \left\{ f(g^d,g^{-b}h^a,\frac{a\tau + b}{d}) \right\}, \]
from which we draw the coefficients of $F_n$.
\end{proof}

A holomorphic function $f$ on $\HH$ is said to satisfy a modular equation of order $n$ if there exists a monic polynomial $F_n(x)$ of degree $n \prod_{p|n} \frac{p+1}{p}$, whose coefficients are polynomials in $f$, and with roots $f(\frac{a\tau+b}{d})$ for $a,b,d$ satisfying $ad=n, 0 \leq b<d, (a,b,d)=1$.  We will use a slightly altered notion to account for invariance under congruence groups other than $\Gamma_0(N)$.

\begin{defn} Let $g,h \in G$ be a commuting pair, and let $f$ be a function on
\[ (\{ (g,g^nh) \}_{n \in \Z} ) \underset{\pm \Z}{\times} \HH. \]
If $p$ is a prime satisfying $g^p=g$ and $h^p=h$, we say $f(g,h,\tau)$ satisfies an equivariant modular equation of order $p$ if there exists a monic polynomial $F_n(x)$ of degree $p+1$, whose coefficients are polynomials in $f(g,h,\tau)$ and whose roots are $f(g, g^{-b}h, \frac{a\tau+b}{d})$ for $a,b,d$ satisfying $ad=p, 0 \leq b<d$.
\end{defn}

When $g=1$ and $f(1,h,\tau)$ has $q$-expansion $q^{-1} + O(q)$, this agrees with the non-equivariant notion.

\begin{prop} \label{prop:Hecke-mod-eq}
Suppose $g,h \in G$ commute.  If a function $f$ is weakly Hecke-monic for $(g,h)$, then $f(g,h,\tau)$ satisfies equivariant modular equations of order $p$ for all primes $p$ congruent to $1$ mod $\lcm(|g|,|h|)$.
\end{prop}
\begin{proof}
Since $g^p = g$ and $h^p = h$, this is a special case of the previous lemma.
\end{proof}

If $f(g,h,\tau)$ satisfies an equivariant modular equation of order $p$, we can write the polynomial $F_p(x)$ as a two-variable polynomial $F_p(y,x) \in \C[x,y]$, where we set $y = f(g,h,\tau)$, and expand the coefficients of $F_p(x)$ as polynomials in $f(g,h,\tau)$.  If $f(g,h,\tau)$ is a nonconstant holomorphic function, $F_p(y,x)$ is uniquely defined by the  properties that it is monic of degree $p+1$ in $x$ and that it vanishes under the substitutions $f(g,h,\tau)$ for $y$ and $f(g, g^{-b}h, \frac{a\tau+b}{d})$ for $x$ for any $\tau \in \HH$.  This is because a polynomial in one variable is uniquely determined by its values on a nonempty open subset of $\C$, and the coefficients of $F_p(x)$ are polynomials in $f(g,h,\tau)$, which includes such an open subset in its range.

\begin{lem}
Let $p$ be a prime satisfying $g^p = g$ and $h^p = h$.  Suppose $f(g,h,\tau)$ is a nonconstant holomorphic function satisfying an equivariant modular equation of order $p$.  Then $F_p(y,x) = F_p(x,y)$.
\end{lem}
\begin{proof}
This is a modification of proposition 3.2 in \cite{K94}.

If $d=1$, then $F_p(f(g,h,\tau), f(g,h,p\tau)) = 0$.  We make the substitution $\tau := \frac{\tau'+b}{p} - b$ for $0 \leq b < p$, and we get 
\[ \begin{aligned}
0 &= F_p\left(f\left(g,h,\frac{\tau'+b}{p}-b\right), f(g,h,\tau'+b-pb)\right) \\
&= F_p\left(f\left(g,g^{-b}h,\frac{\tau'+b}{p}\right), f(g,h,\tau')\right)
\end{aligned} \]
Then $f(g,g^{-b}h,(\tau'+b)/p)$ is a root of $F_p(y, f(g,h,\tau'))$

If $d=p$, then $F_p(f(g,h,\tau),f(g,h,(\tau+b)/p-b)) = 0$.  We make the substitution $\tau = p\tau'+pb-b$ for $0 \leq b < p$, and we get
\[ \begin{aligned}
0 &= F_p(f(g,h,p\tau'+pb-b),f(g,h,\tau')) \\
&= F_p(f(g,h,p\tau'), f(g,h,\tau'))
\end{aligned} \]
Then $f(g,h,p\tau')$ is a root of $F_p(y, f(g,h,\tau'))$

This proves that $F_n(y,f(g,h,\tau))$ has roots $f(g,g^{-b}h,\frac{a\tau+b}{d})$, which means that for any fixed $\tau \in \HH$, $F_p(f(g,h,\tau),x) = F_p(x,f(g,h,\tau)) \in \C[x]$.  The coefficients of $F_p(x) = F_p(x, f(g,h,\tau))$ are polynomials in $f$, so they are uniquely determined by finitely many values.  If $f$ is nonconstant and holomorphic on some nonempty open set, then the coefficients of $F_p(x,y)$ match those of $F_p(y,x)$, so we get a polynomial equality.
\end{proof}

\begin{prop} \label{prop:global}
If $f$ is weakly Hecke-monic for $(g,h)$ and $f(g,h,\tau)$ has a pole at infinity, then $f(g,h,\tau)$ admits global symmetries, i.e., if $f(g,h,\tau_1) = f(g,h,\tau_2)$ for a given $\tau_1,\tau_2 \in \HH$, then there exists $\gamma \in SL_2(\R)$ such that $\tau_1 = \gamma\tau_2$, and $f(g,h,\tau) = f(g,h,\gamma\tau)$ for all $\tau \in \HH$.
\end{prop}
\begin{proof}
By Proposition \ref{prop:Hecke-mod-eq}, $f(g,h,\tau)$ satisfies equivariant modular equations of degree $p$ for infinitely many primes $p$ congruent to $1$ modulo $\lcm(|g|,|h|)$.

We give a list of modifications of the first half of \cite{CG97} (up to Proposition 4.6) to allow equivariance.  Note that the summands for $d=p$ are $f(g,g^{-b}h, \frac{\tau+b}{p}) = f(g,h,\frac{\tau+b}{p}-b)$, $0 \leq b < p$, so we make the global modification that
\[ A(p) = \left\{ \begin{pmatrix} 1 & (1-p)b \\ 0 & p \end{pmatrix} | 0 \leq b < p \right\} \cup \left\{ \begin{pmatrix} p & 0 \\ 0 & 1 \end{pmatrix} \right\} \]

The proof of existence of global symmetries in \cite{CG97} needs the following cosmetic changes:
\begin{enumerate}
\item The statement of Lemma 2.2 condition 1 should be changed from $z_1-z_2 \in \Z$ to $z_1-z_2 \in \frac{|g|}{C}\Z$.
\item Lemma 2.5 requires $\beta$ to change to $\begin{pmatrix} n/d & r-dr \\ 0 & d \end{pmatrix}$.  The proof uses the symmetry of $F_p(x,y)$, proved in the above lemma.
\item The statement of Lemma 3.2 requires the form of $\beta \in A(p)$ to be changed as above.
\item All occurrences of $\Z$ in the proof of Lemma 3.3 should be replaced by $\frac{|g|}{C}\Z$.
\item The phrase ``translating by integers if necessary'' in the proof of Proposition 4.3 should be replaced by ``translating by integer multiples of $\frac{|g|}{C}$ if necessary.''
\item In the proof of Proposition 4.6, the form of $\beta \in A(p)$ needs to be suitably adjusted.
\end{enumerate}
\end{proof}

\begin{cor}
Let $f$ be a Hecke-monic function on $\mathcal{M}_{Ell}^G$.  If $f(g,h,\tau)$ is nonconstant, then it is a holomorphic congruence genus-zero function.
\end{cor}
\begin{proof}
From our hypotheses, we know that $f(g,h,\tau)$ is invariant under some $\Gamma(N)$, and we are assuming that $f$ has no essential singularities at cusps.  Therefore, if $f(g,h,\tau)$ is nonconstant, then there is some $\binom{ab}{cd} \in SL_2(\Z)$ such that $f(g,h, \frac{a\tau + b}{c\tau + d})$ has a pole at infinity.  Then $f(g^ah^c,g^bh^d, \tau)$ has a pole at infinity, and satisfies the hypotheses of the proposition.  This implies $f(g^ah^c,g^bh^d, \tau)$ is holomorphic congruence genus-zero, so $f(g,h,\tau)$ is also.
\end{proof}

\section{Finite level}

Theorem 1.3 in \cite{CG97} asserts that any series $q^{-1} + O(q)$ with algebraic integer coefficients satisfying modular equations of all orders coprime to some $n$ is either holomorphic genus-zero and invariant under some $\Gamma_0(N)$, or a function of the form $q^{-1} + \zeta q$ for $\zeta$ either zero or a 24th root of unity.  The hypotheses we use to prove Theorem \ref{thm:main} are weaker, since the functions satisfy equivariant modular equations only for primes congruent to 1 (mod $n$), and the functions have the form $q^{C/|g|} + O(1)$.  However, our conclusions are weaker, since even if we normalize to an integral-powered $q$-series, we only have invariance under $\Gamma_1(N)$, much like the situation in \cite{C02}.

\begin{defn} Let $G$ be a subgroup of $SL_2(\R)$, and let $M$, $N$, and $C$ be nonzero integers, such that $M|N$.  We say that the quadruple $(G, M, N, C)$ satisfies properties 1-3 if:
\begin{enumerate}
\item $G$ is a discrete group.
\item The stabilizer of infinity $G_\infty \subset G$ is $\langle -Id, \begin{pmatrix} 1 & M \\ 0 & 1 \end{pmatrix} \rangle$.
\item For all primes $p$ congruent to $1$ mod $N$, and all $\binom{ab}{cd} \in G$, there exist integers $l$ and $k$ such that $l|p$, $0 \leq -k < p/l$, and such that:
\[ \begin{pmatrix} \frac{ap}{l} & \frac{k(1-p)a}{C}+lb \\ \frac{c}{l} & \frac{1}{p}\left( \frac{k(1-p)c}{C}+ld\right) \end{pmatrix}  \in G \]
\end{enumerate}
\end{defn}

\begin{lem} If $(G,M,N,C)$ satisfies properties 1-3, and $\gamma \in G$, then there exists $\lambda \in \R$ such that $\lambda \gamma \in GL_2^+(\Q)$.
\end{lem}
\begin{proof}
This is a minor variation of Lemma 5.4 in \cite{CG97}.  Our $G_\infty$ is a subgroup of theirs, so double coset invariants surject.  We define $r_m \equiv a/c$ (mod $M$) instead of mod 1, but it is still a $G_\infty$ double coset invariant for our $G_\infty$.  The proof there uses a slightly different property 3 for $G$, but the two left entries of the matrices match, and that is what was needed.
\end{proof}

Following \cite{CG97}, we say that $\binom{ab}{cd}$ is primitive if $a,b,c,d$ are integers with no common factors.  By the previous lemma, there exists for any $\gamma \in G$ some $\lambda \in \R$ (unique up to sign) such that $\lambda \gamma$ is primitive, and we define $|\gamma|$ to be the determinant of $\lambda \gamma$.  This is an invariant of the double coset $G_\infty \gamma G_\infty$.

\begin{lem} \label{lem:CG5.7}
Let $(G,M,N,C)$ satisfy properties 1-3, let $\gamma_1 \in G$, let $\lambda \gamma_1 = \begin{pmatrix} a_1 & b_1 \\ c_1 & d_1 \end{pmatrix}$ be primitive, and assume $c_1 \neq 0$.  Choose a prime $p \equiv 1$ (mod $NC$), and choose a sequence of elements $\{ \gamma_n \}_{n\geq 1} \subset G$ by iteratively applying property 3.  Define $\lambda \gamma_i = \begin{pmatrix} a_i & b_i \\ c_i & d_i \end{pmatrix} \in M_2(\Q)$, and let $l_i, k_i \in \Z$ be the corresponding integers arising in each application of property 3.  Then
\begin{enumerate}
\item The sequence $\{ c_i \}_{i \geq 1}$ eventually stabilizes to some $c_\infty = c_1/\prod_{i \geq 1} l_i \in \Q$, i.e., all but finitely many $l_i$ are equal to $1$.
\item If $c_\infty \in \Z$, then $d_i \in \Z$ for all $i \geq 1$.
\item If $c_\infty$ is a nonzero integer multiple of $p$, then $p$ divides $d_i$ for all $i \geq 1$.
\item There exists $W>0$, depending only on $c_1$ and $\lambda$, such that if $p>W$, then $l_i=1$ and $d_i \in \Z$ for all $i \geq 1$.
\end{enumerate}
\end{lem}
\begin{proof}
This is a minor alteration of Lemma 5.7 in \cite{CG97}, and we will point out the necessary changes.

The first statement follows from Lemma 1.25 of \cite{S71}, which asserts that the lower left entries of elements of a discrete subgroup of $SL_2(\R)$ that don't fix infinity are bounded away from zero.

The second and third statements can be proved by following the proofs of Lemma 5.7a and 5.7b in \cite{CG97}, and changing $n$ to $p$, $k_i$ to $k_i(1-p)/C$ (which is an integer by our assumption on $p$), and $p^{2((j-i)\eta+s)+s'}$ to $p^{2(i-i_0 + s)+s'}$.  The last alteration is mostly to rectify a typographical error.

The assertion about $l_i = 1$ in the fourth statement follows from Lemma 1.25 of \cite{S71}, and the assertion about $d_i$ follows from the second statement.
\end{proof}

\begin{lem}
Suppose $(G,M,N,C)$ satisfies properties 1-3, and $G$ does not stabilize infinity.  Then $G$ contains an element of the form $\binom{10}{n1}$ for $n$ a nonzero multiple of $NC$.
\end{lem}
\begin{proof}
Since $G$ does not stabilize infinity, then $G$ contains $\gamma$ such that the primitive $\lambda \gamma = \binom{ab}{cd}$ has $c \neq 0$, and hence
\[ \gamma' = \gamma \begin{pmatrix} 1 & -NC|\gamma| \\ 0 & 1 \end{pmatrix} \gamma^{-1} = \begin{pmatrix} 1 + NCac & -NCa^2 \\ NCc^2 & 1-NCac \end{pmatrix} \in G \cap \Gamma(N). \]

Let $W(\gamma')$ be the constant given by the fourth part of lemma \ref{lem:CG5.7}.  After translating on the right by multiples of $\begin{pmatrix} 1 & NC \\ 0 & 1 \end{pmatrix}$, we find that $G$ has an element $g = \begin{pmatrix} 1 + NCac & b' \\ NCc^2 & p \end{pmatrix}$, with $p \equiv 1$ (mod $NC$) a prime larger than $W(\gamma')$.  Since both matrices have primitive multipliers $\lambda = 1$ and the same bottom left entries, $W(\gamma') = W(g)$.  We apply the fourth part of lemma \ref{lem:CG5.7} and use property 3 to find that $G$ contains $\begin{pmatrix} (1+NCac)p & k(1-p)(1-NCac)/C + b' \\ NCc^2 & k(1-p)Nc^2 + 1 \end{pmatrix} \in \Gamma(NC)$ for some $0 \leq k < p$.   Since $M|N$, we can multiply on the right by $\begin{pmatrix} 1 & kmN \\ 0 & 1 \end{pmatrix} \in G$, where $m = \frac{p-1}{NC} \in \Z$, and this yields $\begin{pmatrix} * & * \\ NCc^2 & 1 \end{pmatrix} \in \Gamma(N)$.  We multiply on the left by a suitable multiple of $\begin{pmatrix} 1 & N \\ 0 & 1 \end{pmatrix}$ to get $\begin{pmatrix} 1 & 0 \\ NCc^2 & 1 \end{pmatrix} \in G$.
\end{proof}

\begin{lem}
Let $X$ be a set of matrices $\begin{pmatrix} 1+an & bn \\ cn & 1+dn \end{pmatrix} \in \Gamma(n)$, satisfying:
\begin{enumerate}
\item For every integer $c_0$, there exists an element of $X$ as above with $c = c_0$.
\item For all nonzero $c$, and all $a_0$ and $d_0$ satisfying $(1+a_0 n,cn) = (1+d_0 n,cn) = 1$, there exists an element of $X$ as above such that $a \equiv a_0$ (mod $|c|n$) and $d \equiv d_0$ (mod $|c|n$).
\end{enumerate}
Then $X$ is a complete set of double coset representatives for $\Gamma(n)$ with respect to the subgroup $\langle \begin{pmatrix} 1 & n \\ 0 & 1 \end{pmatrix} \rangle$.
\end{lem}
\begin{proof}
\[ \begin{aligned}
&\begin{pmatrix} 1 & en \\ 0 & 1 \end{pmatrix} \begin{pmatrix} 1+an & bn \\ cn & 1+dn \end{pmatrix} \begin{pmatrix} 1 & fn \\ 0 & 1 \end{pmatrix} = \\
&\qquad =\begin{pmatrix} 1+an+cen^2 & (b+e+f)n + (af+de)n^2 + cefn^3 \\ cn & 1+dn+cfn^2 \end{pmatrix}
\end{aligned} \]
To find all double coset representatives, it suffices to cover the possible lower triangular entries, since for $c \neq 0$, the top right entry is uniquely determined by the fact that the determinant is one.  Any element of $X$ satisfying $c=0$ lies in the group $\langle\begin{pmatrix} 1 & n \\ 0 & 1 \end{pmatrix} \rangle$, so the corresponding double coset is equal to this group.
\end{proof}

\begin{lem}
Suppose $(G,M,N,C)$ satisfies properties 1-3, and suppose $G$ contains $\binom{10}{n1}$ for some nonzero integer $n$ for $NC|n$.  Then $G$ contains $\Gamma(n)$.
\end{lem}
\begin{proof}
It suffices to produce double coset representatives with respect to translations, so suppose we are given $a,c,d$ satisfying the conditions in the above lemma, with $c \neq 0$.  Let $r>0$ be a lower bound on absolute value of nonzero lower left entries of elements of $G$, guaranteed by Lemma 1.25 of \cite{S71}.  By Dirichlet, there exist primes $p$ and $q$ such that $p \equiv 1 + an$ (mod $|c|n^2$), $q \equiv 1 + dn$ (mod $|c|n^2$), and $p > \text{max}(|c|n/r, |c|n)$.  Since $(1+an)(1+dn)-bcn^2 = 1$, there exists an integer $m$ such that $pq = mcn^2 + 1$.  Then:
\[ \begin{pmatrix} 1 & 0 \\ cn & 1 \end{pmatrix} \begin{pmatrix} 1 & mn \\ 0 & 1 \end{pmatrix} = \begin{pmatrix} 1 & mn \\ cn & pq \end{pmatrix} \in G \]
By the fourth statement in Lemma \ref{lem:CG5.7}, any application of property 3 for our choice of $p$ to this matrix requires $l=1$, so $G$ contains $\begin{pmatrix} p & k(1-p)p/C + mn \\ cn & k(1-p)cn/pC + q \end{pmatrix}$ with $\frac{k(1-p)cn}{pC} \in \Z$.  By our assumptions on $p$, $(1-p)cn$ is coprime to $p$, so $k=0$.  Therefore, $G$ contains $\begin{pmatrix} p & mn \\ cn & q \end{pmatrix}$, and this is the desired double coset representative.
\end{proof}

We say that a function on $\HH$ is of trigonometric type if after some transformation $\tau \mapsto a \tau + b$, it has the form $q^{-1} + a_0 + \zeta q$, for $\zeta$ a root of unity or zero.

\begin{thm} \label{thm:main}
Let $f$ be weakly Hecke-monic for $(g,h)$, and suppose $f(g,h,\tau)$ has a pole at infinity, and $q$-expansion coefficients that are algebraic integers.  Then $f(g,h,\tau)$ is either of trigonometric type or holomorphic congruence genus-zero.
\end{thm}
\begin{proof}
By Proposition \ref{prop:expansion}, the $q$-expansion of $f(g,h,\tau)$ has the form $\zeta q^{C/|g|} + O(1) \in \overline{\Q}((q^{-C/|g|}))$ for some root of unity $\zeta$ and some negative integer $C$.  By Proposition \ref{prop:Hecke-mod-eq}, $f(g,h,\tau)$ satisfies equivariant modular equations for all primes $p$ satisfying $g^p = g$, $h^p = h$.  Following the proof of \cite{CG97} Lemma 7.1 (changing $\frac{az+b}{d}$ to $\frac{az+b(1-d)}{d}$ and $\Z$ to $\frac{|g|}{C}\Z$), we find that $f(g,h,\tau)$ is invariant under a discrete subgroup of $SL_2(\R)$.  By proposition \ref{prop:global}, $f(g,h,\tau)$ admits global symmetries, and in particular, an altered version of \cite{CG97}, Lemma 3.2 holds, where $A(p)$ is replaced by the equivariant version.  In summary, the group $G$ of global symmetries of $f(g,h,\tau)$ satisfies the following three conditions:

\begin{enumerate}
\item $G$ is a discrete group.
\item The stabilizer of infinity $G_\infty \subset G$ is $\langle -Id, \begin{pmatrix} 1 & \frac{|g|}{C} \\ 0 & 1 \end{pmatrix} \rangle$.
\item For all primes $p$ congruent to $1$ mod $\lcm(|g|,|h|)$, and all $\binom{ab}{cd} \in G$, there exist integers $l$ and $k$ such that $l|p$, $0 \leq -k < p/l$, and such that:
\[ \begin{pmatrix} p & 0 \\ 0 & 1 \end{pmatrix} \begin{pmatrix} a & b \\ c & d \end{pmatrix} \begin{pmatrix} l & k(p-1) \\ 0 & p/l \end{pmatrix}^{-1} = \begin{pmatrix} ap/l & k(1-p)a+lb \\ c/l & (k(1-p)c+ld)/p \end{pmatrix}  \in G \]
\end{enumerate}

We now consider the function $f(g,h,-\tau/C)$, which is a power series in $q^{1/|g|}$.  Let $G'$ denote the subgroup of $SL_2(\R)$ that fixes $f(g,h,-\tau/C)$, so $G' = \begin{pmatrix} -1/C & 0 \\ 0 & 1 \end{pmatrix} G \begin{pmatrix} -C & 0 \\ 0 & 1 \end{pmatrix}$.  The quadruple $(G', |g|, \lcm(|g|,|h|),C)$ then satisfies conditions 1-3.

If $G'$ does not fix infinity, then by the previous lemmata, $G'$ contains $\Gamma(n)$ for some $n$, and $G$ contains some congruence group.  $f(g,h,\tau)$ is therefore a holomorphic congruence genus-zero function.

If $G'$ fixes infinity, then $G = G_\infty = \langle -Id, \begin{pmatrix} 1 & \frac{|g|}{C} \\ 0 & 1 \end{pmatrix} \rangle$.  The proof that $f(g,h,\tau)$ is of trigonometric type is given by following the first half of the proof of \cite{CG97}, Lemma 7.2, and replacing modular equations with equivariant modular equations, and $q$ with $q^{-C/|g|}$.
\end{proof}

\noindent\textbf{Remark:} The proof of \cite{CG97}, Lemma 7.2 allows us to make a slightly stronger statement: If we don't necessarily have algebraic integer coefficients, but we know that $f(g,h,\tau) = \zeta q^{C/|g|} + \sum_{n \geq 0} a_n q^{-nC/|g|}$ satisfies $\sum n|a_n|^2 > 1$, then $f(g,h,\tau)$ is a holomorphic congruence genus-zero function.

\section{Replicability}

In this paragraph, we summarize some results and assertions in \cite{N84}.  One can start with a formal power series $f(q) = q^{-1} + \sum_{n > 0} a_n q^n$ and define numbers $H_{m,n}$ for $m,n \in \Z_{>0}$ by the bivarial transform:
\[ \log \frac{f(p)-f(q)}{p^{-1}-q^{-1}} = - \sum_{m,n=1}^\infty H_{m,n} p^m q^n.\]
For each $n>0$ there is a unique normalized Faber polynomial $\Phi_n(x)$ (depending on $f$), defined by the property that $\Phi_n(f(q)) = q^{-n} + O(q)$.  Since the polynomial in $f(q)$ that is $n$ times the coefficient of $p^n$ in the formal series $- \log p(f(p) - f(q))$ also has this form, we have $\Phi_n(f(q)) = q^{-n} + n\sum_m H_{m,n}q^m$.  We shall say that $f(q)$ is replicable if and only if for any $t \in \Z_{>0}$ there exists a series $f^{(t)}(q) = q^{-1} + O(q)$ such that $\Phi_n(f(q)) = \sum_{ad=n, 0 \leq b<d} f^{(a)}(\frac{a\tau + b}{d})$.  The series $f^{(t)} = \sum_{n>0} a_n^{(t)} q^n$ is called the $t$-th replicate of $f$, and by suitable use of induction, one can show that it is unique if it exists, and its coefficients satisfy the relation:
\[ H_{m,n} = \sum_{t|(m,n)} \frac{1}{t} a_{mn/t^2}^{(t)} \]
In particular, if $f$ is replicable, $H_{m,n}$ only depends on $(m,n)$ and $mn$.  Another induction argument implies the converse of this, i.e., that one can define replicability by this independence.

\noindent\textbf{Note:} Replicability was originally defined only for power series with rational integer coefficients, and for more general series, there is some disagreement in the literature regarding the correct definition.  Norton has proposed a definition of replicability for series that have irrational cyclotomic integer coefficients, and it seems to involve a Galois action.  This is potentially useful when considering functions invariant under some $\Gamma_0(N)$.  One can instead extend one of the integral definitions above to allow arbitrary complex coefficients, without alteration of the formulas, and this was done in \cite{K94}.  We will use the latter generalization of the condition in this paper, because we can prove something about series that satisfy it.


\begin{defn} A replicable function is a replicable power series that converges on the open unit $q$-disc, i.e., one that expands to a holomorphic function on $\HH$.  A replicable function has order $n$ if $f^{(m)} = f^{(m+n)}$ for all $m>0$.
\end{defn}

We note that a replicable function of finite order has a unique minimal order, but not a unique order.  If $f$ is a replicable function, then all $\Phi_m(f)$ are holomorphic on $\HH$, hence all $f^{(m)}$ are also holomorphic on $\HH$.

We would like to relate replicability to Hecke-monicity.

\begin{lem} \label{lem:unique}
If $f$ is a weakly Hecke-monic function for $(1,g)$, such that $f(1,g,\tau) = q^{-1} + O(q)$, then $f(1,g^m,\tau)$ has the form $q^{-1} + O(1)$, and is uniquely defined by the Hecke-monic property up to a constant.
\end{lem}
\begin{proof}
For the purposes of induction, we assume $f(1,g^k,\tau) = q^{-1} + O(1)$ for all $k<m$.  Then
\[ \begin{aligned}
mT_mf(1,g,\tau) &= \underset{0 \leq b<d}{\sum_{ad=m}} f\left(1,g^a, \frac{a\tau+b}{d}\right) \\
&= f(1,g^m,m\tau) + \sum_{d|m, d < m} \sum_{0 \leq b < d} e(b/d) q^{m/d^2} + O(1) \\
\end{aligned} \]
Since $mT_mf(1,g,\tau)$ is monic of degree $m$ in $f(1,g,\tau)$, the leading term is $q^{-m}$, and all of the other summands have poles of lower order.  By subtracting those summands, we find that the leading term of $f(1,g^m,m\tau)$ is $q^{-m}$, so $f(1,g^m,\tau)$ has leading term $q^{-1}$.  Since $f(1,g^m,\tau)$ is a power series in $q$, it has the form we want.

To show uniqueness, suppose there were some $f'(1,g^m,\tau) = q^{-1} + O(1)$ such that $f'(1,g^m,m\tau) + \sum_{ad=m, d < m, 0 \leq b < d} f(1,g^d,\frac{a\tau+b}{d})$ is monic of degree $m$ in $f(1,g,\tau)$.  Since this sum and $mT_mf(1,g,\tau)$ have the same coefficients in negative degree, $f'(1,g^m,m\tau) - f(1,g^m,m\tau) = O(1)$.  However, this difference must be a polynomial in $f(1,g,\tau)$, so it is constant.
\end{proof}

\begin{lem} \label{lem:Tn-form}
If $f$ is a weakly Hecke-monic function for $(1,g)$, such that $f(1,g^m,\tau) = q^{-1} + O(q)$ for all $m>0$, then $nT_nf(1,g,\tau)$ is the unique polynomial in $f(1,g,\tau)$ whose expansion is $q^{-n} + O(q)$.
\end{lem}
\begin{proof}
\[ \begin{aligned}
nT_nf(1,g,\tau) &= \underset{0 \leq b<d}{\sum_{ad=n}} f\left(1,g^a,\frac{a\tau+b}{d}\right) \\
&= \sum_{d|n} \sum_{0\leq b <d} e\left(\frac{-(n/d)\tau-b}{d}\right) + O(q^{1/n}) \\
&= \sum_{d|n} e(-n\tau/d^2) \sum_{0 \leq b < d} e(-b/d) + O(q^{1/n})\\
&= \sum_{d|n} e(-n\tau/d^2) \delta_{d,1} + O(q^{1/n}) \\
&= q^{-n} + O(q^{1/n})
\end{aligned} \]
Since $f(1,g,\tau)$ is a power series in $q$ and $nT_nf(1,g,\tau)$ is a polynomial in $f(1,g,\tau)$, we can refine the $O(q^{1/n})$ to $O(q)$.  If we add any other polynomial in $f(1,g,\tau)$, the leading term will yield a nontrivial contribution to the nonpositive powers in the expansion, so the polynomial is unique.
\end{proof}

\begin{prop}
The map $f^{(m)}(\tau) \mapsto f(1,g^m,\tau)$ induces a bijection between replicable functions of order $N$ and weakly Hecke-monic functions for $(1,g)$ on $\Hom(\Z \times \Z, \Z/N\Z) \underset{\pm \Z}{\times} \HH$ whose expansions at infinity have the form $q^{-1} + O(q)$, where $g$ is a generator of $\Z/N\Z$.
\end{prop}
\begin{proof}
We first assume that $f^{(1)}$ is replicable, so $\Phi_n(f^{(1)}) = \sum_{ad=n,0\leq b < d} f^{(a)}\left(\frac{a\tau + b}{d}\right)$.  Then
\[ \begin{aligned}
\Phi_n(f^{(1)}(\tau)) &= \underset{0 \leq b<d}{\sum_{ad=n}} f^{(a)}\left(\frac{a\tau+b}{d}\right) \\
&= \underset{0 \leq b<d}{\sum_{ad=n}} f\left(1, g^a, \frac{a\tau + b}{d}\right) \\
&= nT_nf(1,g,\tau)
\end{aligned} \]
Therefore, $nTnf(1, g, \tau)$ is a monic polynomial in $f(1,g,\tau)$ for all $n$.  $f$ is therefore weakly Hecke-monic for $(1,g)$.

Now, let $f$ be a weakly Hecke-monic function for $(1,g)$ satisfying $f(1,g^i,\tau) = q^{-1} + O(q)$.  By Lemma \ref{lem:Tn-form}, $mT_mf(1,g,\tau) = q^{-m} + O(q)$, and is a monic polynomial in $f(1,g,\tau)$, so it is equal to $q^{-m} + m\sum_k H_{k,m}q^k = \Phi_n(f)$.  If we assume for the purposes of induction that $f(1,g^k,\tau) = f^{(k)}(\tau)$ for all $k|m$, $k \neq m$, then
\[ \begin{aligned}
\underset{0 \leq b<d}{\sum_{ad=m}} f\left(1,g^a,\frac{a\tau+b}{d}\right) &= mT_mf(1,g,\tau) \\
&= \Phi_m(f) \\
&= \underset{0 \leq b<d}{\sum_{ad=m}} f^{(a)}\left(\frac{a\tau+b}{d}\right)
\end{aligned} \]
implies $f^{(m)} = f(1,g^m,\tau)$.
\end{proof}

\begin{cor}
If $f$ is a replicable function of finite order with algebraic integer coefficients, then $f$ is either of trigonometric type or holomorphic congruence genus-zero for a group containing $\Gamma_1(N)$ for some $N$.
\end{cor}
\begin{proof}
By the above proposition, $f$ together with its replicates forms a weakly Hecke-monic function for $(1,g)$, where $g$ generates a cyclic group whose order is that of $f$.   By Theorem \ref{thm:main}, $f(1,g,\tau)$ is either of trigonometric type or holomorphic congruece genus-zero and invariant under a group containing $\Gamma(N)$ for some $N$.  Since $f$ is invariant under translation by $1$, it is invariant under $\Gamma_1(N)$.
\end{proof}

Norton also defined a stronger notion in \cite{N84}: $f$ is completely replicable if all $f^{(t)}$ are replicable, or equivalently, if the $s$-th replication power of $f^{(t)}$ is $f^{(st)}$ for all $s$ and $t$ (see Proposition 2.5 in \cite{K94}).  He also pointed out that $-J(z+1/2) = q^{-1} + 196884q - 21493760q^2 + \dots$ is a function that is replicable but not completely replicable.

\begin{cor}
The above bijection specializes to a bijection between completely replicable functions of order $N$ and semi-weakly Hecke-monic functions for $(1,g)$ on $\Hom(\Z \times \Z, \Z/N\Z) \underset{\pm \Z}{\times} \HH$ whose expansions at infinity have the form $q^{-1} + O(q)$.
\end{cor}
\begin{proof}
Using the proposition, we get a chain of equivalent statements:
\begin{enumerate}
\item $f^{(1)}$ is completely replicable.
\item $f^{(m)}$ is replicable for all $m$.
\item $f$ is weakly Hecke-monic for all $(1,g^m)$.
\item $f$ is semi-weakly Hecke-monic for $(1,g)$.
\end{enumerate}
\end{proof}

\begin{cor}
The above bijection specializes to a bijection between completely replicable functions $f^{(1)}$ with rational integer coefficients invariant under $\Gamma_0(N)$ and Hecke-monic functions $f$ on $\M^{\Z/N\Z}_{Ell}$ satisfying the property that the $q$-expansions of $f(1,g^i,\tau)$ have the form $q^{-1} + O(q)$, with rational integer coefficients.
\end{cor}
\begin{proof}
It suffices to show that if $f^{(1)}$ is invariant under $\Gamma_0(N)$, then $f^{(m)}$ is invariant under $\Gamma_0(N/(m,N))$.  The completely replicable functions with integer coefficients were exhaustively enumerated in \cite{ACMS92}, and their fixing groups were found to obey this condition in \cite{F93}.
\end{proof}

Replicable functions without a specified order also have an interpretation in terms of Hecke-monicity, if we allow our group $G$ to be infinite.  If we let $g$ generate a copy of $\Z$, we can think of replicable functions together with their replicable powers as weakly Hecke-monic functions for $(1,g)$ on $\Hom(\Z \times \Z, \Z) \underset{\pm \Z}{\times} \HH$.  Unfortunately, the finite order condition is essential for our techniques to produce a genus-zero statement.

\section{Twisted denominator formulas}

Given a Lie algebra of a rather specialized form described below, we can make strong statements about certain characters of automorphisms acting on homology.  When this Lie algebra arises from conformal field theory in a certain way, we show that in fact the characters are holomorphic congruence genus-zero functions.  The particular constraints on the Lie algebra force it to be ``mostly free'' in the sense that its higher homology is very small.  This is somewhat related to work  \cite{J98} on free Lie subalgebras of generalized Kac-Moody algebras like the monster Lie algebra.  Some connections to elliptic cohomology appear in unpublished work of Lurie concerning exponential operations on elliptic $\lambda$-rings \cite{L05}.

Let $G$ be a finite group, and let $g$ be an element of order $N$ in the center of $G$.  Suppose we have a collection $\mathcal{V} = \{ V^{i,j/N}_k | i, j \in \Z/N\Z, k \in \frac{1}{N}\Z \}$ of $G$-modules, satisfying the following properties:
\begin{itemize}
\item The action of $g$ on $V^{i,j/N}_k$ is given by constant multiplication by the root of unity $e(j/N)$.
\item $\dim V^{i,j/N}_k$ grows subexponentially with $k$, i.e., for any $\epsilon>0$, there is some $C>0$ such that for all $i,j,k$, we have $\dim V^{i,j/N}_k < Ce^{\epsilon k}$.
\end{itemize}

\noindent\textbf{Note:} We will occasionally write $V^{i,j/N}_k$ where $i$ and $j$ are given as integers.  This means we are tacitly reducing modulo $N$, so $V^{i,j/N}_k$ is the same $G$-module as $V^{i+aN,j/N+b}_k$ for all integers $a$ and $b$.

\begin{defn}
A complex Lie algebra $E$ is Fricke compatible with $\mathcal{V}$ if:
\begin{enumerate}
\item $E$ is graded by $\Z_{> 0} \times \frac{1}{N}\Z$, with finite dimensional homogeneous components $E_{i,j}$.  We introduce degree indicator symbols $p$ and $q$, which denote grading shifts by $(1,0)$ and $(0,\frac{1}{N})$, respectively, and write the graded vector space decomposition as $E = \bigoplus_{i>0, j \in \frac{1}{N}\Z} E_{i,j} p^i q^j$.  We can view this as a character decomposition of $E$ under an action of a two dimensional torus $H$.
\item $E$ admits a homogeneous action of $G$ by Lie algebra automorphisms, such that we have $G$-module isomorphisms $E_{i,j} \cong V^{i,j}_{1+ij}$
\item The homology of $E$ is given by:
\begin{itemize}
\item $H_0(E) = \C$
\item $H_1(E) = \bigoplus_{n \in \frac{1}{N}\Z} V^{1,n}_{1+n} pq^n$
\item $H_2(E) = p \bigoplus_{m=0}^\infty V^{1,-1/N}_{1-1/N} \otimes V^{m,1/N}_{1+m/N} p^m$
\item $H_i(E) = 0$ for $i>2$.
\end{itemize}
\item $E_{1,-1/N} \cong V^{1,-1/N}_{1-1/N}$ is one dimensional.
\end{enumerate}
\end{defn}

\noindent\textbf{Remark:} Our use of the term ``Fricke compatible'' is motivated by considerations from conformal field theory.  If $g$ is a Fricke element of the monster, i.e., if the McKay-Thompson series $T_g(\tau) = \mathrm{Tr}(gq^{L_0 - 1}|V^\natural)$ is invariant under the transformation $\tau \mapsto -1/N\tau$ for some $N$, and if $G$ is a central extension of the centralizer of $g$ in the monster, then we expect the Lie algebra of physical states of the $g$-orbifold intertwining algebra to be a generalized Kac-Moody algebra whose positive subalgebra is isomorphic to $E$ as a Lie algebra with a homogeneous action of $G$ by automorphisms, and we expect the unique irreducible $g$-twisted module of $V^\natural$ to be isomorphic to $H_1(E)$ as a graded $G$-module.  If $g$ is a non-Fricke element, then we expect the compatible Lie algebra to have a large abelian subalgebra and higher homology described by its exterior powers.  We explore this further in \cite{C09} and \cite{C10}.

\begin{prop} (Twisted denominator formula) Suppose $E$ is Fricke compatible with $\mathcal{V}$.  Then for any $h \in G$,
\[ \begin{aligned}
p^{-1} &+ \sum_{m > 0} \mathrm{Tr}(h|V^{1,-1/N}_{1-1/N}) \mathrm{Tr}(h|V^{m,1/N}_{1+m/N}) p^m - \sum_{n \in \frac{1}{N}\Z} \mathrm{Tr}(h|V^{1,n}_{n+1}) q^n \\
&= p^{-1} \exp \left( - \sum_{i > 0} \sum_{m > 0, n \in \frac{1}{N}\Z} \mathrm{Tr}(h^i | V^{m,n}_{1+mn})p^{im}q^{in}/i \right)
\end{aligned} \]
\end{prop}
\begin{proof}
This is essentially identical to section 8 in \cite{B92}.  The Chevalley-Eilenberg resolution yields the equation $H(E) = \bigwedge(E)$ of virtual $H \times G$-representations, and the left side is given by taking traces on the homology groups given above.  By Adams' exponential formula from $K$-theory, we have
\[ \bigwedge(U) = \exp\left(-\sum_{i>0} \frac{\psi^i(U)}{i}\right) \]
for any finite dimensional $H \times G$-module $U$ (which we take to be the homogeneous components $E_{i,j}$ or finite sums thereof).  The $\psi^i$ are the $i$th Adams operations, which satisfy the identity $\mathrm{Tr}(g|\psi^i(U)) = \mathrm{Tr}(g^i|U)$.  The right side of the equation is then given by extending this to a formal sum on the infinite dimensional direct sum of homogeneous components, and this is allowed because their degrees are supported in a strict half-space.
\end{proof}

For any $h \in G$, we define ``formal orbifold partition functions'':
\[ Z(g^k,g^lh^m,\tau) := \sum_{n \in \frac{1}{N}\Z} \underset{n \in kr+\Z}{\sum_{r \in \frac{1}{N}\Z/\Z}} \text{Tr}(g^lh^m|V^{k,r}_{1+n}) e(n\tau) \]
We refer to the collection of these functions as $Z$, and they converge on $\HH$, by the subexponential growth condition.  We can then define equivariant Hecke operators: 
\[ T_n Z(g,h,\tau) = \frac{1}{n} \underset{0 \leq b<d}{\sum_{ad=n}} Z\left(g^d, g^{-b}h^a, \frac{a\tau+b}{d}\right) \]
\begin{prop}
Suppose $E$ is Fricke compatible with $\mathcal{V}$.  Then $Z$ is weakly Hecke-monic for $(g,h)$.
\end{prop}
\begin{proof}
We multiply both sides of the twisted denominator formula by $p$, and viewing the equality as an identification of formal expansions, we take logarithms.

\[ \begin{aligned}
 \log &\left(1 - p \sum_{n \in \frac{1}{N}\Z} \text{Tr}(h|V^{1,n}_{1+n}) q^n + \sum_{m > 0} \text{Tr}(h|V^{1,-1/N}_{1-1/N}) \text{Tr}(h|V^{m,1/N}_{1+m/N}) p^{m+1} \right) \\
&= - \sum_{i > 0} \sum_{m > 0, n \in \frac{1}{N}\Z} \text{Tr}(h^i | V^{m,n}_{1+mn})p^{im}q^{in}/i \\
&= - \sum_{m > 0} \sum_{a|m} \frac{1}{a} \sum_{n \in \frac{1}{N}\Z} \text{Tr}(h^a | V^{m/a,n}_{1+mn/a}) p^m q^{an} \\
&= - \sum_{m > 0} \sum_{ad=m} \frac{1}{a} \sum_{0 \leq b < d} \frac{1}{d} \sum_{n \in \frac{1}{N}\Z} \text{Tr}(h^a | V^{d,n}_{1+dn}) p^m q^{an} \\
&= - \sum_{m > 0} \sum_{ad=m} \frac{1}{a} \sum_{0 \leq b < d} \frac{1}{d} \sum_{n \in \frac{1}{N}\Z} \underset{n \in dr+\Z}{\sum_{r \in \frac{1}{N}\Z/\Z}} e(-br) \text{Tr}(h^a | V^{d,r}_{1+n}) e(br) q^{an/d} p^m \\
&= - \sum_{m > 0} \frac{1}{m} \underset{0 \leq b < d}{\sum_{ad=m}} \sum_{n \in \frac{1}{N}\Z} \underset{n \in dr+\Z}{\sum_{r \in \frac{1}{N}\Z/\Z}} \text{Tr} (g^{-b}h^a | V^{d,r}_{1+n}) e\left(n\frac{a\tau + b}{d}\right) p^m \\
&= -\sum_{m > 0} \frac{1}{m} \underset{0 \leq b < d}{\sum_{ad=m}} Z\left(g^d, g^{-b}h^a,\frac{a\tau+b}{d}\right) p^m \\
&= -\sum_{m > 0} T_m Z(g,h, \tau) p^m
\end{aligned} \]

Isolating the terms that are degree $k$ in $p$ on the first line yields a polynomial of degree $k$ in $Z(g,h,\tau)$, with leading coefficient $-1/k$.  This implies that for all $k$, $kT_kZ(g,h,\tau)$ as a formal $q$-series is a monic polynomial of degree $k$ in the $q$-expansion of $Z(g,h,\tau)$.  All of the formal orbifold partition functions uniquely define holomorphic functions on $\HH$, so for all $k$, $kT_kZ(g,h,\tau)$ is a monic polynomial of degree $k$ in $Z(g,h,\tau)$, where they are viewed as functions on $\HH$.
\end{proof}

We describe a connection to generalized moonshine.  Recall that one of the key hypotheses in the conjecture was the existence of certain representations of central extensions of centralizers of elements.  An interpretation of these representations was given in \cite{DGH88}, where they were said to be twisted Hilbert spaces of an orbifold conformal field theory.  In our language, these are twisted modules of the vertex operator algebra $V^\natural$.  The theoretical details of vertex operator algebras and twisted modules are outside the scope of this paper, but we can think of these objects as graded vector spaces, where the grading is given by eigenvalues of a semisimple operator $L_0$.  When a vertex operator algebra has a unique irreducible $g$-twisted module for some automorphism $g$, Schur's Lemma produces a natural action of some central extension of the centralizer of $g$ on the twisted module.  The two facts we need concerning twisted modules are from \cite{DLM00}:
\begin{enumerate}
\item (Theorem 10.3) If $V$ is a holomorphic $C_2$-cofinite vertex operator algebra with central charge 24, and $g$ is a conformal automorphism of finite order, then there exists a unique irreducible $g$-twisted module $V(g)$ up to isomorphism.
\item (Theorems 5.4, 6.4, and 8.1) Let $M>0$ satisfy $g^M = h^M = 1$, and suppose we have a $G$-module isomorphism
\[ V(g^i) \cong \bigoplus_{k \in \frac{1}{N}\Z} \bigoplus_{j \in \Z/N\Z} V^{i,j/N}_k, \]
where $V(g^i)$ is the irreducible $g^i$-twisted module, and the outer sum gives the $L_0$-eigenvalue decomposition.  For any $\binom{ab}{cd} \in SL_2(\Z)$, $Z\left(g^i,h,\frac{a\tau+b}{c\tau+d}\right)$ lies in a certain space of holomorphic functions on $\HH$, each element of which is annihilated by a differential operator of the form $\left(\frac{d}{d\tau}\right)^m + \sum_{j=0}^{m-1} r_j(q) \left(\frac{d}{d\tau}\right)^j$, where $m > 0$ and $r_j(q) \in \C[[q^{1/M}]]$ converges on $\HH$.
\end{enumerate}

\begin{prop}
Suppose that $E$ is a Lie algebra Fricke compatible with $\mathcal{V}$, and suppose that $G$ acts conformally on a holomorphic $C_2$-cofinite vertex operator algebra $V$ of central charge 24, such that for all $i \in \Z/N\Z$, we have $G$-module isomorphisms $V(g^i) \cong \bigoplus_{k \in \frac{1}{N}\Z} \bigoplus_{j \in \Z/N\Z} V^{i,j/N}_k$ as in Fact \#2.  Then $Z(g,h,\tau)$ is a holomorphic congruence genus-zero function.
\end{prop}
\begin{proof}
By the previous proposition, $Z$ is weakly Hecke-monic for $(g,h)$.  Since $E_{1,-1/N} = V^{1,-1/N}_{1-1/N}$ is one-dimensional, the trace of $h$ on this space is nonzero, so $Z(g,h,\tau)$ has a pole at infinity.  By Theorem \ref{thm:main}, $Z(g,h,\tau)$ is then either a holomorphic congruence genus-zero function, or of trigonometric type.  However, the expansion of any function of trigonometric type at a cusp other than infinity is not annihilated by any differential operator of the form given in Fact \#2 above.
\end{proof}

The hypotheses for this proposition are quite strong, but it is not a vacuous statement.  When $G = \MM$ and $g = 1$, this implies the McKay-Thompson series are holomorphic congruence genus-zero modular functions, assuming the positive subalgebra of the monster Lie algebra is Fricke compatible with $V^\natural$.  This compatibility was proved in section 8 of \cite{B92}.  When $G =2.B$, the nontrivial central extension of the baby monster simple group, and $g$ is the central element of order two, this yields holomorphic congruence genus-zero characters for the conjugacy class 2A case of generalized moonshine, assuming there exists a Lie algebra Fricke compatible with the suitable twisted modules.  The holomorphic congruence genus-zero result was proved in \cite{H03} using a construction of a Fricke compatible Lie algebra, and the above proposition allows one to eliminate the explicit computations in the final step of the proof, which involved matching the first 25 coefficients of the character for every conjugacy class of $G$ with Norton's list of known replicable functions.

\end{document}